%% file: l1tv.tex
\documentclass[11pt]{article}

\usepackage[utf8]{inputenc}
\usepackage[T1]{fontenc}

\usepackage{amsmath, amsfonts,amsthm}
\usepackage{bbm}
\usepackage{txfonts}

\usepackage[font=small,labelfont=bf]{caption}

\usepackage[a4paper]{geometry}

\usepackage{color}

\input{symbols}

\input{customcommands}

\usepackage[sort&compress, numbers]{natbib}

\newtheorem{theorem}{Theorem}
\theoremstyle{remark}

\usepackage{aliascnt}
\newaliascnt{conj}{theorem}
\newaliascnt{cor}{theorem}
\newaliascnt{lemma}{theorem}
\newaliascnt{prop}{theorem}
\newaliascnt{definition}{theorem}
\newaliascnt{example}{theorem}
\newaliascnt{notation}{theorem}
\newaliascnt{experiment}{theorem}

\theoremstyle{theorem}

\theoremstyle{definition}

\aliascntresetthe{conj}
\aliascntresetthe{cor}
\aliascntresetthe{lemma}
\aliascntresetthe{prop}
\aliascntresetthe{definition}
\aliascntresetthe{example}
\aliascntresetthe{notation}
\aliascntresetthe{experiment}

\usepackage{snapshot}

\usepackage{graphicx}
\usepackage{epstopdf}
\usepackage{booktabs}
\usepackage{multicol}

\usepackage[boxed, ruled, lined]{algorithm2e}

\usepackage[usenames,dvipsnames,svgnames]{xcolor}

\usepackage{enumerate}
\usepackage{csquotes}

\usepackage{subcaption}

\usepackage[pdftex, 
 colorlinks = true,
  linkcolor= black,
  citecolor = black]{hyperref}

\newcommand{\val}{\mathrm{Val}}

\newcommand{\T}{\mathbb{T}}

\newcommand{\deltaSNR}{\mathrm{\Delta SNR}}

\title{ 
Exact algorithms for $L^1$-TV regularization 
of real-valued or circle-valued signals
 }
\author{Martin Storath\thanks{Biomedical Imaging Group, École Polytechnique Fédérale de Lausanne, Switzerland.}, Andreas Weinmann\thanks{Department of Mathematics, Technische Universität München, and Helmholtz Zentrum M\"unchen, Germany.}, Michael Unser\footnotemark[1]
}
\date{\today}

\begin{document}
	
\newlength\figureheight
\newlength\figurewidth
\setlength\figureheight{0.15\textwidth}

\maketitle

\begin{abstract}
\normalsize
We consider $L^1$-TV regularization
of univariate signals with values on the real line or on the unit circle.
While the real data space
leads to a convex optimization problem, 
the problem is non-convex for circle-valued data.
In this paper, we derive exact 
algorithms for both data spaces.
A key ingredient is the reduction of 
the infinite search spaces 
to a finite set of configurations,
which can be scanned  
by the Viterbi algorithm.
To reduce the computational complexity of the involved
tabulations, we extend 
the technique of distance transforms to non-uniform grids
and to the circular data space.
In total, the proposed algorithms have complexity $\Oc(KN)$
where $N$ is the length of the signal 
and $K$ is the number of different values in the data set.
In particular, 
the complexity is $\Oc(N)$ for quantized data. 
It is the first exact  algorithm
for TV regularization with circle-valued data,
and it is competitive with the state-of-the-art methods
for scalar data, assuming that the latter are quantized.
\end{abstract}


\section{Introduction}

Total variation (TV) minimization has
become a standard method for  jump or edge preserving
regularization of signals and images.
Whereas the classical $L^2$-TV model
(i.e., TV with quadratic data fidelity term \cite{rudin1992nonlinear})
is optimally matched to the Gaussian noise model,
$L^1$ data terms are more robust
to noise with more heavy tailed distributions 
such as Laplacian noise, and to the presence of outliers; 
see, e.g., \cite{nikolova2002minimizers}.
Further advantages 
are the better preservation of the contrast
and the invariance to global contrast changes
\cite{chan2005aspects}.
Since $L^1$-TV minimization is
a convex problem for
real- and vector-valued data,
it is accessible by convex optimization techniques.  
In fact, there are several 
algorithms for $L^1$-TV minimization with scalar and vectorial data.
The minimization methods are typically of iterative nature:
for example, interior point methods \cite{fu2006efficient},
iterative thresholding \cite{beck2009fast}, alternating methods of multipliers \cite{wang2008new, goldstein2009split},
semismooth Newton methods \cite{clason2009duality},
primal-dual strategies \cite{dong2009efficient, chambolle2011first}, and proximal point methods \cite{micchelli2013proximity}.
There are also other  algorithms based on recursive median filtering \cite{alliney1997property}
or graph cuts \cite{darbon2006image}.

For univariate real-valued signals, 
 efficient exact algorithms are available
for $L^2$-TV; for instance
the taut string algorithm which has a linear complexity \cite{mammen1997locally,davies2001local}.
A recent alternative is the algorithm of Condat  \cite{condat2012direct}
which shows a particularly good performance in practice.
The $L^1$-TV problem is computationally more intricate.
For data $y \in \R^N$ and
a non-negative weight vector $w \in \R^N,$ it is given by
\begin{equation}\label{eq:l1tv} 
	\argmin_{x \in \R^N} \ \alpha \sum_{n=1}^{N-1} |x_{n} - x_{n+1}| + \sum_{n=1}^N w_n | x_n - y_n|,
\end{equation}
where $\alpha>0$ is a model parameter regulating the tradeoff between data fidelity
and TV prior.
In a Bayesian framework,
it corresponds to the maximum a posteriori estimator
of a summation process with Laplace distributed increments
under a Laplacian noise model; see, e.g., \cite{unser2014introduction}.
Kovac and Dümbgen \cite{dumbgen2009extensions} 
have derived an exact solver of complexity $\Oc(N \log N)$
for \eqref{eq:l1tv}.
Recently, Kolmogorov et al.~\cite{kolmogorov2015total} have proposed a solver of complexity
$\Oc(N \log\log N).$ 

Recently, total variation 
regularization on non-vectorial data spaces such as, e.g., Riemannian manifolds
has received a lot of interest \cite{chan2001total, CS13,WDS2013, lellmann2013total, grohs2014total, baust2015total}.
The non-vectorial setting is a major challenge 
because the total variation problem 
is, in general, not  anymore convex.
One of the simplest examples, 
where the $L^1$-TV functional is nonconvex, is circle-valued data. 
Such data appears, for 
example, as phase signals (which are defined modulo $2\pi$) 
and as time series of angles.
Particular examples for the latter are  
the data on the orientation of the bacterial flagellar motor \cite{sowa2005direct}
and the data on wind directions  \cite{davis2002statistics}.
The $L^1$-TV functional for circle-valued data $y \in \T^N$ is given by
\begin{equation}\label{eq:tvCirc} 
	\argmin_{x \in \T^N} \ \alpha \sum_{n=1}^{N-1} d_\T(x_{n}, x_{n+1}) + \sum_{n=1}^N w_n\, d_\T(x_n, y_n),
\end{equation}
where $d_\T(u,v)$ denotes the arc length distance of $u, v \in \T = \S^1.$
Theoretical results on 
total cyclic variation can be found  
in the papers of Giaquinta et al.~\cite{giaquinta93variational}
and of Cremers and Strekalovskiy \cite{CS13}.
The authors of the latter one have shown that the problem is computationally 
at least as complex as the Potts problem; 
this means, in particular, that it is NP-hard in dimensions greater than one.
Current minimization strategies for \eqref{eq:tvCirc}
are based on convex relaxations \cite{CS13},
proximal point splittings \cite{WDS2013},
or iteratively reweighted least squares \cite{grohs2014total}.
However, due to the non-convexity of \eqref{eq:tvCirc}, 
these iterative approaches 
do not guarantee convergence to a global minimizer.
Furthermore, they are computationally  demanding. 
To our knowledge, 
no exact algorithm for \eqref{eq:tvCirc} has been proposed yet.

In this paper, we propose 
exact non-iterative 
algorithms for $L^1$-TV minimization
on scalar signals \eqref{eq:l1tv}  and on circle-valued signals \eqref{eq:tvCirc}.
A key ingredient is
the reduction of the infinite search space,
$\R^N$ or $\T^N,$ to a finite search space $V^N.$
This reduction 
allows us to use the Viterbi algorithm \cite{viterbi1967error,forney1973viterbi}
for the minimization of discretized energies 
as presented in  \cite{felzenszwalb2011dynamic}.
A time-critical step in the Viterbi algorithm 
is the computation of a distance transform
w.r.t. the non-uniform grid induced by $V.$
For the scalar case, we generalize the efficient 
two-pass algorithm of Felzenszwalb and Huttenlocher 
\cite{felzenszwalb2004distance,felzenszwalb2006efficient}
from uniform grids to our non-uniform setup.
We further propose a new method for efficiently computing the distance
transforms in the circle-valued case.
In total, our solvers have complexity $\Oc(KN)$
where $K$ denotes the number of 
different values in the data.
In particular, if the 
data is quantized to finitely many levels, 
the algorithmic complexity is $\Oc(N).$
It  is the first exact algorithm 
for TV regularization of circle-valued signals,
and it is competitive with the state-of-the-art methods for 
real-valued signals, assuming that the latter are quantized.

\begin{figure}
	\def\figfolder{experiments/randomSignal/}
		\def\figurewidth{0.47\columnwidth}
	\centering	
\centering
\includegraphics[width=\figurewidth]{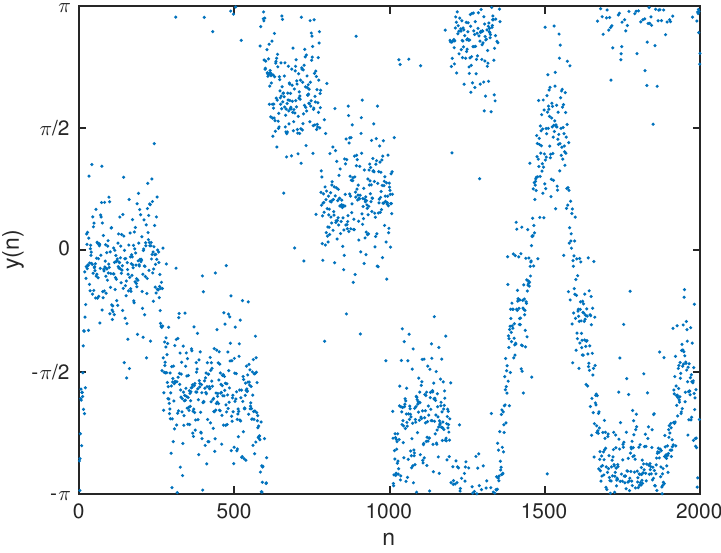} 
\hfill
\includegraphics[width=\figurewidth]{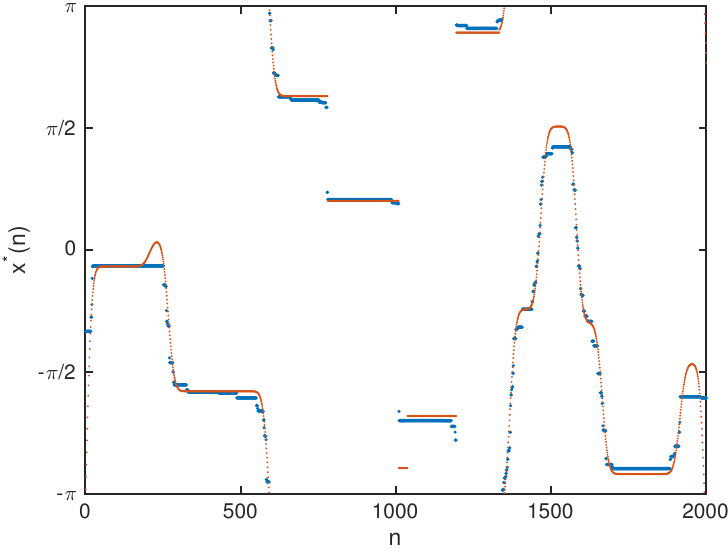} 
	\caption{\emph{Left:} Synthetic circle-valued signal corrupted by noise.   
	\emph{Right:} Global minimizer $x^*$ of the TV  functional with 
	$\alpha = \protect\input{\figfolder alpha.txt}.$ 
	(Ground truth displayed as small red points.)
	The noise is almost completely removed and the jumps are preserved.
	The phase jumps of $2\pi$ are taken into account properly.
	}
\label{fig:synthCirc}
\end{figure}

\subsection{Organization of the paper}
In Section~\ref{sec:reduction},
we show that the search space
can be reduced to a finite set.
In Section \ref{sec:dynProg}, we 
present our minimization strategy
for the reduced problem.
In Section~\ref{sec:numerics}, 
we present numerical experiments based
on synthetic and real data.
Eventually, we discuss the relations
to other approaches.

\section{Reduction of the search space}\label{sec:reduction}

A crucial step in our derivation
is the reduction of the search space to a finite set.
In the following, we denote the 
$L^1$-TV functional for data $y \in \R^N$ or $y \in \T^N$  by
\[
	T_{\alpha; y}(x) =  \alpha \sum_{n=1}^{N-1} d(x_{n}, x_{n+1}) + \sum_{n=1}^N w_n\, d(x_n, y_n).
\]
Here $d$ denotes the distance
that corresponds to the data space, i.e.,
the Euclidean distance for real-valued data and the arc length distance for circle-valued data.
We  further use the notation $\val(y)$ to denote the set of values of the $N$-tuple $y,$ i.e.,
\[
	\val(y) = \{ v : \text{ there is }n \text{ with }  1\leq n \leq N  \text{ s.t. } y_n = v \},
\]
or, using set notation, $\val(y) = \{y_1, \ldots, y_n\}.$
Also recall that a (weighted) median of $y$
is a minimizer of the functional 
\[
\mu \mapsto \sum_{n=1}^N w_n d(\mu, y_n).
\]
We will see that there are always minimizers of the 
$L^1$-TV problem whose values are all contained
in the values $\val(y)$ of the data~$y$ (united with the antipodal points $\val(\tilde y)$ in the circle-valued case).

\subsection{Real-valued data}

Let us first consider the real-valued case.
The following 
assertion on the minimizers of the $L^1$-TV functional
has been proven by Alliney \cite{alliney1997property}.
His proof is based on results of convex analysis.
Here, we develop an alternative technique 
which does not exploit the convexity of the TV functional.
The crucial point is that this technique will allow us to
treat the more involved non-convex circular case later on.

\begin{theorem}\label{thm:setOfMinimizersScal}
Let $\alpha >0, $ $y \in \R^N,$ and $V = \val(y).$
Then 
\[
	\min_{x\in \R^N} T_{\alpha; y}(x) = \min_{x\in V^N} T_{\alpha; y}(x).
\]
\end{theorem}

\begin{proof}
The method of proof is as follows: 
we consider an arbitrary $x \in \R^N$
and construct $x' \in V^N$ such 
that $T_{\alpha;y}(x') \leq T_{\alpha;y}(x).$
If we apply this procedure to a minimizer $x^\ast,$
we obtain a minimizer with values in $V^N$  
which is the assertion of the theorem. 

So let $x \in \R^N$ be arbitrary
and let us construct $x' \in V^N$
with smaller or equal $T_{\alpha;y}$ value
by the following procedure.
Let $\Ic$ be the set of maximal intervals of $\{1,\ldots,N\}$
where $x$ is constant on and where $x$ does not attain its value in $V.$
It means that each element $I$ of $\Ic$ is an \enquote{interval} of the form $I = \{l, l+1, \ldots, r\}$ 
such that $a := x_l = \ldots = x_r  \notin V$ and such that $x_{l-1} \neq x_l$ (unless $l=1$)
and $x_{r} \neq x_{r+1}$ (unless $r=N$).
If $\Ic$ is empty then $x \in  V^N$ and we are done.
Otherwise,   we decrease the number of such intervals $|\Ic|$
by the following rule:
Choose an interval 
$I = \{l,\ldots, r\} \in \Ic.$
We construct $\bar{x}$ which equals $x$ 
outside $I$ and choose its constant value $a'$ such that the 
corresponding number of intervals with values which are not in $V$ is strictly smaller than $|\Ic|.$ 
We distinguish three cases.
 
First assume that $I$ is not a boundary interval (i.e. $l\neq 1$ and $r\neq N$) and that the values 
of the two neighboring intervals, $x_{l-1}$ and $x_{r+1},$
are both 
smaller than the value on $I$ which is equal to $a.$ (We call an interval a neighbor of $I$ if it contains the index $l-1$ or $r+1.$)
We denote the nearest smaller and the nearest greater neighbors 
of $a$ in $V$ by $b^-$ and $b^+,$ respectively.
Let $b' = \max(x_{l-1}, x_{r+1}, b^-).$
By replacing $a$ by  some $a' \in [b', b^+]$
we change the total variation penalty  by $2\alpha(a'-a)$
and the data penalty by $(W^- - W^+) (a'-a).$
Here $W^+ = \sum_{\{i \in I: y_i > a\}} w_i$ and $W^-= \sum_{\{i \in I: y_i \leq a\}} w_i$ 
are the weights 
of elements in the interval $I$ that are greater 
or smaller than $a,$ respectively.
If $W^- + 2\alpha < W^+,$ 
we let $a'$ equal its greater neighbor $b^+.$
Otherwise, we let $a' = b'$ where, by the definition above, 
$b' = \max (x_{l-1}, x_{r+1}, b^-).$
If $a' = b^-,$ then the value of $\bar x$ on $I$ belongs to $V.$
If $a' \in \{ x_{l-1}, x_{r+1} \}$ the interval merges with one of its neighbors.
In both cases, the number of intervals with \enquote{undesired} values,  $|\Ic|,$ decreases by one.
By symmetry, the same argumentation is 
valid for the case that the 
values of the neighboring intervals 
$x_{l-1}$ and $x_{r+1}$ are both 
greater than~$a.$

As second case we consider the situation where $I$ is not a boundary interval, 
and where $x_{l-1}$ is smaller and $x_{r+1}$ is greater than 
$a.$ (Again, the case $x_{l-1}>a>x_{r+1}$ is dealt with by symmetry.)
Since replacing $a$ by any value in $[x_{l-1}, x_{r+1}]$
does not change the total variation penalty,
we only need to look at 
the approximation error. 
This amounts to setting $a'$
equal to a (weighted) median of $y_{l}, \ldots, y_r.$
Note that there exists a (weighted) median that it is contained in $\{y_{l}, \ldots, y_r\} \subset V.$ We use such a median in $V$ to define $a'.$
Hence, also in this case, $|\Ic|$ decreases by one.

Finally, we consider the third case where
the interval is located at the boundary.
If either $1 \in I$ or $N \in I$
then we proceed analogously to the first case.
The relevant difference is that
we let $a'$ equal its greater nearest neighbor $b^+$
if $W^- + \alpha < W^+$ (instead of $W^- + 2\alpha < W^+$). 
If the interval touches both boundaries, i.e., if $I = \{1,\ldots,N\},$
we proceed as in the second case,
which is setting $a'$ to be a (weighted) median $y$
which is contained in $V.$

We repeat the above procedure until 
$|\Ic| = 0$ 
which implies that the final result $x'$ is contained in $V^N.$
By construction, the functional value $T_{\alpha;y}(x')$ is not exceeding the 
functional value of $x,$ since all intermediately constructed $\bar x$ do so.
This completes the proof.
\end{proof}

Note that the assertion of Theorem~\ref{thm:setOfMinimizersScal}
is not true for quadratic data fidelities.
As the following simple example shows,
it is not uncommon
that $\val(\hat x) \cap \val(y) = \emptyset$
for all $L^2$-TV minimizers~$\hat x.$
We consider toy data $y = (0,1) \in \R^2$
and the corresponding $L^2$-TV functional given by
$x\mapsto \alpha |x_1 - x_2| + x_1^2 + (x_2 - 1)^2.$
It is easy to check that the unique minimizer 
of this $L^2$-TV problem is given by 
$\hat x = (\alpha/2, 1 - \alpha/2),$ if $\alpha < 1,$
and by  $\hat x = (1/2, 1/2),$ otherwise.
We note that this is an example
where $\val(\hat x) \cap \val(y) = \emptyset$
even for all $\alpha > 0.$
This shows that one cannot even expect an analogous result when one chooses a suitable parameter.
For a more detailed 
discussion of this aspect we refer to the paper of Nikolova \cite{nikolova2002minimizers}.
It is interesting to note that
an assertion analogous to that of Theorem~\ref{thm:setOfMinimizersScal} 
can be shown for the Potts model,
although the model and the corresponding proof are quite different;
see \cite{weinmann2015l1potts, storath2015jump}.

\subsection{Circle-valued data}

Now we use the  techniques developed for the real-valued case in our proof
of Theorem~\ref{thm:setOfMinimizersScal} in the more involved situation
of circle-valued data to prove the following theorem allowing for the 
reduction of the search space for minimizers of the $L^1$-TV functional 
for $\mathbb S^1$-valued data as well.

\begin{theorem}\label{thm:setOfMinimizersCirc}
Let $\alpha >0, $ $y \in \T^N,$ and $V = \val(y) \cup \val(\tilde y),$
where $\tilde y$ denotes the tuple of antipodal points of $y.$ 
Then 
\[
	\min_{x\in \T^N} T_{\alpha; y}(x) = \min_{x\in V^N} T_{\alpha; y}(x).
\]
\end{theorem}

\begin{proof}
As in the proof of Theorem \ref{thm:setOfMinimizersScal},
we consider an arbitrary $x \in \T^N$
and construct $x' \in V^N$ such 
that $T_{\alpha;y}(x') \leq T_{\alpha;y}(x).$
Note that, in contrast to the proof of Theorem \ref{thm:setOfMinimizersScal},  
$V = \val(y) \cup \val(\tilde y)$ here.
Similarly, we
let $\Ic$ be the set of the maximal intervals $I$ of $\{1,\ldots,N\}$
where $x$ is constant on and where the attained value $a$ of $x$ on $I$ 
is not contained in $V.$
We decrease the number of such intervals $|\Ic|$
by the procedure explained below. 

Before being able to give the explanation we need some notions 
related to $\mathbb S^1$ data. Let us consider a point $a$ on the 
sphere and its antipodal point $\tilde a.$ Then there are two hemisphere/half-circles connecting $a$ and $\tilde a.$ We use the convention that $\tilde a$ is contained 
in both hemispheres whereas $a$ is contained in none of them.  
These two hemispheres
can be distinguished into the hemisphere $H_1=H_1(a)$ determined by walking from $a$
in clockwise direction and the hemisphere $H_2=H_2(a)$ obtained from walking in counter-clockwise direction.

Equipped with these preparations, we explain
the procedure to reduce the number of intervals $|\Ic|.$
We pick an arbitrary interval $I = \{l,\ldots, r\} \in \Ic$
and let $a = x_l = \ldots = x_r$
be the value of $x$ on $I.$
We let $b_1$ and $b_2$ be the nearest neighbors
of $a$ in $H_1 \cap V$ and in $H_2 \cap V,$
which are the values  of the data (or their antipodal points) on the clockwise and counter-clockwise hemisphere, respectively.
We note that $b_1,b_2$ exist and both are not equal to the antipodal point $\tilde a$ of $a.$
This is because, together with a point $p$, its antipodal point $\tilde p$
is also contained in $V$ which implies that either $p$ or $\tilde p$ is a member of $H_1$
and either $\tilde p$ or $p$ is a member of $H_2.$ Since $a$ is not contained in the set $V$
of values of $y$ and its antipodal points, the distance to either $p$ or $\tilde  p$ is strictly smaller that $\pi.$  
We construct $\bar{x}$ which equals $x$ 
outside $I$ and with constant value $a'$ on $I$ such that $|\Ic|$ decreases. 
We have to differentiate three cases.

First we assume that $I$ is no boundary interval
and that the left and the right neighboring candidate item $x_{l-1}$ and $x_{r+1}$ are both located on the clockwise hemisphere $H_1$ and none of them agrees with $\tilde a$.
Let $W_1 = \sum_{i: y_i \in H_1} w_i$ be the weight of $y$ on $H_1$
and let $W_2 = \sum_{i: y_i \in H_2} w_i$ be the weight of $y$ on $H_2.$
(Note that $\tilde a$ which is the only point in both $H_1$ and $H_2$ is not a member of $y$.)
If $W_1 > W_2 + 2\alpha,$
which means that the clockwise hemisphere $H_1$
is \enquote{heavier} than the counterclockwise hemisphere $H_2$ plus the variation penalty,
we set $a'$ to be the nearest neighbor
of $a$ in $\{x_l, x_r, b_1\}.$ 
This may be visualized as shifting the value on $I$ in clockwise direction 
until we hit the first value in $\{x_l, x_r, b_1\}.$ 
Since $W_1 > W_2 + 2\alpha,$ we have that 
$T_{\alpha; y}(\bar x) \leq T_{\alpha; y}(x).$ 
Otherwise, we set $a' = b_2$
which means that we shift to the other direction.
Since then $W_1 \leq W_2 + 2\alpha,$ we get
$T_{\alpha; y}(\bar x) \leq T_{\alpha; y}(x)$ also in this situation.
By symmetry, the same argument applies when both 
$x_{l-1}$ and $x_{r+1}$ are located on the counterclockwise hemisphere.

In the second case we assume that $I$ is no boundary interval
and that 
$x_{l-1}$ and $x_{r+1}$ are located on different hemispheres.
Here we also include the case where one or both 
$x_{l-1}$ and $x_{r+1}$ are antipodal to $a.$
If only one neighbor is antipodal, we interpret it to lie on the opposite hemisphere
of the non-antipodal member.
If both neighbors are antipodal, we interpret them to lie on different hemispheres.
We let $\mathcal C$ be the arc connecting $x_{l-1}$ and $x_{r+1}$
which has $a$ as member.
Letting $a'$ equal any value on the arc $\mathcal C,$
leads to $TV(\bar x) \leq TV(x),$
meaning that it does not increase the variation penalty 
$TV(x) = \sum_n \alpha d_\T(x_n, x_{n+1})$.
By definition, the data term is minimized by letting $a'$ 
be a (weighted) median of $y_l, \ldots, y_r.$
A (weighted) median of the circle-valued data can be chosen 
as an element of the unique values
$\{y_l, \ldots, y_r\}$ unified with the antipodal points
$\{\tilde y_l, \ldots, \tilde y_r\}.$
We choose $a'$ as such a median. This implies
$T_{\alpha; y}(\bar x) \leq T_{\alpha; y}(x).$

It remains to consider the boundary intervals.
If $I = \{1,\ldots,N\},$
we proceed as in the second case
and set $\bar u_i =  a'$ for all $i,$ where $a'$ is a (weighted) median of $y$
which is contained in $V.$
Else, if either $1 \in I$ or $N \in I$ we proceed analogously to the first case
with the difference that we replace the decision criterion $W_1 > W_2 + 2\alpha$
employed there by $W_1 > W_2 + \alpha.$

We repeat the above procedure until 
$|\Ic| = 0$ which implies that the values of the final result $x'$
all lie in $V^N.$
Then plugging in a minimizer $x=x^\ast,$
results in a minimizer $x' \in V^N$  
which shows the theorem. 
\end{proof}

As for scalar data, the assertion of Theorem~\ref{thm:setOfMinimizersCirc}
is not true for quadratic data terms.
This can be seen using the previous example
interpreting the data $y = (0,1)$ as angles.

In order to illustrate the difference to the real-valued data case,
let us point out a degenerate situation which is due to the circular nature of the data. 
Assume that the data only  consists
of a point $z \in \T$ and its antipodal point $\tilde z,$ i.e., $y = (z, \tilde z).$
For sufficiently large $\alpha$, 
any minimizer $\hat x$ of \eqref{eq:tvCirc} is constant; 
say $\hat x = (a,a).$ 
Since the TV penalty gets equal to zero, 
$a$ must be equal to a median  of $y.$
It is not hard to check that 
every point on the sphere is a median of $y.$
This behavior appears curious at first glance.
However, 
the data shows no clear tendency towards a distinguished orientation.
Thus, every estimate can be considered as equally good.
The result seems even more
natural than that of $L^2$-TV regularization.
An $L^2$-TV minimizer would consists of
one of the two \enquote{mean orientations}
which are given by rotating $z$ by $\pi/2$
in clockwise or counterclockwise direction.
Both minimizers seem rather arbitrary,
and,  moreover, the two options point into opposing directions.

\section{Efficient algorithms for the reduced problems}\label{sec:dynProg}
Theorem~\ref{thm:setOfMinimizersScal} and Theorem~\ref{thm:setOfMinimizersCirc}
allow us to reduce the infinite search spaces $\R^N$ and $\T^N$ 
in \eqref{eq:l1tv} and \eqref{eq:tvCirc}, respectively,
to the finite sets $V^N,$ which are specified 
in these theorems.
Thus, it remains to solve the problems: find
$$
x^\ast \in  \argmin_{x \in V^N}  T_{\alpha; y}(x).
$$
This can be achieved with dynamic programming 
whose basic idea is to decompose the problem into a series of similar, simpler and tractable subproblems.
For an early account on dynamic programming, we refer to \cite{bellman1957dynamic}.

\subsection{The Viterbi algorithm for energy minimization on finite search spaces}
\label{subsec:GeneralViterbi}

We utilize a dynamic programming scheme
developed by Viterbi \cite{viterbi1967error}; see also \cite{forney1973viterbi}.
Related algorithms have been proposed in \cite{bellman1969curve,blake1987visual}.
In this paragraph, we review a special instance of the Viterbi algorithm
following the presentation of the 
survey  \cite{felzenszwalb2011dynamic}.

We aim at minimizing an energy functional of the form
\begin{equation}\label{eq:generalEnergy}
	 E(x_1,\ldots, x_N) = \alpha \sum_{n=1}^{N-1} d(x_n,x_{n+1}) + \sum_{n=1}^N w_n d(x_n, y_n)
\end{equation}
where the arguments $x_1,\ldots, x_N$ can take values in a finite set $V = \{ v_1, \ldots, v_K\}.$
The Viterbi algorithm solves this problem in  two steps:
tabulation of energies and reconstruction by backtracking.

For the tabulation step,
the starting point is the table $B^1 \in \R^K$ given by
\[
	B^1_k = w_1 d(v_k, y_1) \quad\text{for }k = 1,\ldots, K.	
\]
From now on, the symbol $K$ denotes the cardinality of $V.$  
For $n=2,\ldots,N$ we successively compute the tables $B^n \in \R^K$
which are 
given  
 by
\begin{equation}\label{eq:bellmanEquation}
	B^n_k = w_n d(v_k, y_n) + \min_{l} \{ B^{n-1}_l + \alpha\, d(v_k,v_l)\}, 
\end{equation}
for  $k = 1,\ldots,K.$
The entry $B^n_k$ represents the energy 
of a minimizer on data $(y_1,\ldots,y_n)$ whose endpoint is equal to  $v_k.$

For the backtracking step,
it is convenient to introduce an auxiliary tuple $l \in \N^N$
which stores minimizing indices.
We initialize the last entry of $l$ by
	$l_N = \argmin_{k} B^N_k.$
Then we successively compute the entries of $l$ for $n = N-1, N-2, \ldots,1$  by
\begin{equation}
	l_n = \argmin_{k} B^{n}_k
	 + \alpha\, d(v_k,v_{l_{n+1}}).
\end{equation}
Eventually, we reconstruct a minimizer $\hat x$ 
from the indices in $l$  by
\[
	\hat x_n = v_{l_n}, \quad\text{for }n = 1,\ldots, N. 
\]
The result $\hat x$ is a global 
minimizer of the energy \eqref{eq:generalEnergy}; see \cite{felzenszwalb2011dynamic}.
For a general functional, filling the table $B^n$ in \eqref{eq:bellmanEquation} costs  $\Oc(K^2).$  
This implies that the described procedure is in $\Oc(K^2 N).$ 
In the next subsections, we will derive procedures to reduce the
complexity for filling the tables $B^n$ for our concrete problem to $\Oc(K).$

\subsection{Distance transform on a non-uniform real-valued grid}
\label{subsec:DistTrafoReal}

We first consider the case of real-valued data.
The time critical part of the Viterbi algorithm is
the computation of the minima 
\begin{equation}\label{eq:distTransReal}
D_k = \min_{l} B_l + \alpha |v_k - v_l|, \quad\text{for all $k = 1, \ldots, K.$}
\end{equation}
This problem is known as 
distance transform with respect 
to the $\ell^1$ distance (weighted by $\alpha$).  
Felzenszwalb and Huttenlocher \cite{felzenszwalb2004distance, felzenszwalb2006efficient}
describe an efficient algorithm 
for \eqref{eq:distTransReal} 
when $V$ forms an integer grid, i.e., $V = \{ 0,\ldots, K-1\}.$
In our setup,  $V$ forms a non-uniform grid in general.
Therefore, we generalize their method accordingly.

In the following, we identify the elements of
$V$ with a $K$-dimensional vector $v$ which is 
ordered in ascendingly, i.e., $v_1< v_2 < \ldots < v_K.$ 
The sorting causes no problems since we can sort $v$ in $\Oc(K \log K),$
and since the logarithm of the number of values $K$ is smaller than the data 
length $N,$ we have $\Oc(K \log K) \subset \Oc(K N).$

As we will show below, the following two-pass procedure computes the real-valued distance transform $D$:  

\begin{algorithm}[H]
\caption{Real-valued distance transform {distTransReal($B$, $v$, $\alpha$)}.}
\label{alg:distTrans}
\small
\SetKwInOut{Input}{Input}
\SetKwInOut{Output}{Output}
\SetKwInOut{Local}{Local}
\SetCommentSty{text}
\SetCommentSty{itshape}
\KwIn{$B \in  \R^K;$ $v \in \R^K$ sorted in ascending order; $\alpha > 0;$
}
\KwOut{Distance transform $D$} 
\Begin{
	$D \leftarrow B$\; 
	\For{$k \leftarrow 2, 3, \ldots, K$}{ 
		$D_k \leftarrow \min ( 	D_{k-1}  + \alpha(v_k - v_{k-1}); D_k)$\;
	}
	\For{$k \leftarrow K-1, K-2, \ldots, 1$}{ 
		$D_k \leftarrow \min ( 	D_{k+1}  + \alpha(v_{k+1} - v_k); D_k)$\;
	}
	\textbf{return} $D$\;
}
\end{algorithm}

\noindent In order to show the correctness of the method, we build on 
the structurally related proof given in \cite{felzenszwalb2004distance}
for uniform grids. The major new idea is to pass from discrete 
to continuous infimal convolutions in order to deal with the nonequidistant grid.
The (continuously defined) infimal convolution of two functions $F$ and $G$ on $\R$
with values on the extended real line $[-\infty, \infty]$
is given by
\[
	F \square G(r) = \inf_{u \in \R} \{ F(u) + G(r-u)  \},
\]
see \cite[Section 5]{rockafellar1970convex}.
In the following, 
the infimum will be always attained, so that it actually is a minimum; we use this fact in the notation we employ. 

For real valued data, we get the following result accelerating the bottleneck operation in the general Viterbi algorithm from Section~\ref{subsec:GeneralViterbi}.
\begin{theorem}\label{thm:distTransReal}
	Algorithm~\ref{alg:distTrans} computes  \eqref{eq:distTransReal} in $O(K).$
\end{theorem}
\begin{proof}
We define the function $F$ on $\R$ by
	$F(v_l) = B_l$ for $v_l \in V$ and by  $F(r) = \infty$ for $r \in\R\setminus V$.
Also define $G(u) = \alpha|u|.$
Then, $D_k$ can be formulated in terms of the infimal convolution
of $F$ and $G$ evaluated at $v_k,$ that is,
$$
	D_k = F \square G(v_k).
$$
In order to decompose $G,$ we define 
\[
	G_+(r) = 
	\begin{cases}
		\alpha r, &\text{for }r \geq 0,\\
		\infty, &\text{otherwise,}
	\end{cases}
\quad \text{and} \quad
	G_-(r) = 
	\begin{cases}
		-\alpha r, &\text{for }r \leq 0,\\
		\infty, &\text{otherwise}.
	\end{cases}
\]
We see that $G$ is the infimal convolution of $G_+$ and $G_-$ by using that 
\[
G_+ \square G_-(r) = \min_{t \in \R} { G_+(t) + G_-(r-t)  } = \alpha|r| = G(r).
\]
By the associativity of the infimal convolution (see \cite[Section 5]{rockafellar1970convex}),
we obtain
\begin{equation}\label{eq:assocInfConv}
	F \square G = F \square (G_+ \square G_-) = (F \square G_+) \square G_-.
\end{equation}
We use the right-hand representation; for the right-hand term in brackets,
we get, for $v_k \in V,$  
\begin{align*}
	F \square G_+(v_k) 
	&= \min_{j} { F(v_j) + G_+(v_k-v_j)  } \\
	&= \min_{j \leq k} { F(v_j) + \alpha(v_k-v_j)  } \\
	&=  \min \{ \min_{j \leq k-1}  F(v_j) + \alpha(v_k-v_j)  ; F(v_k)\} \\
		&=  \min \{ \min_{j \leq k-1}  F(v_j) + \alpha(v_{k-1} -v_j + v_k - v_{k-1})  ; F(v_k)\} \\
	&=  \min \{ 	F \square G_+(v_{k-1})  + \alpha(v_k - v_{k-1})  ; F(v_k)\}.
\end{align*}
Now, we denote the result by $F' = F \square G_+$ and continue to manipulate 
the right-hand term of \eqref{eq:assocInfConv} noticing that,
for all $r \notin V,$ we have $F'(r) = \infty.$ 
We obtain 
\begin{align*}
	F' \square G_-(v_k) 
	=& \min_{j} { F'(v_j) + G_-(v_k-v_j)  } \\
	=& \min_{j \geq k} { F'(v_j) - \alpha(v_k-v_j)  } \\
	=&  \min \{ \min_{j \geq k+1}  F'(v_j) - \alpha(v_k-v_j)  ; F'(v_k)\} \\
		=&  \min \{ \min_{j \geq k+1}  F'(v_j) - \alpha(v_{k+1} -v_j + v_k - v_{k+1})  ; F'(v_k) \} \\
	=&  \min \{ 	F \square G_-(v_{k+1})  + \alpha(v_{k-1} - v_k)  ; F'(v_k)\}.
\end{align*}
The above recursive equations show that the forward pass and the backward
pass of Algorithm~\ref{alg:distTrans} compute the desired infimal convolutions.
\end{proof}

\subsection{Distance transform on a non-uniform circle-valued grid}
Now we look at the circular case.
In this case, the corresponding $\ell^1$ distance transform is given by 
\begin{equation}\label{eq:distTransCirc}
D_k = \min_{l} B_l + \alpha d_\T(v_k, v_l), \quad\text{for all $k = 1, \ldots, K.$}
\end{equation}
Our task is to compute the distance transform in the circle case as well.
To this end, we 
employ the angular representation of values on the circle 
in the interval 
 $(-\pi, \pi].$
 As in the real-valued case, we identify 
 the elements of $V$ with a $K$-tuple $v$ which is sorted in ascending order.
In order to compute \eqref{eq:distTransCirc}, we use the following algorithm:

\begin{algorithm}[H]
\caption{Circle-valued distance transform {distTransCirc($B$, $v$, $\alpha$)}.}
\label{alg:distTransCirc}
\small
\SetKwInOut{Input}{Input}
\SetKwInOut{Output}{Output}
\SetKwInOut{Local}{Local}
\SetCommentSty{text}
\SetCommentSty{itshape}
\KwIn{$B \in  \R^K;$ $v \in (-\pi, \pi]^K$ sorted in ascending order; $\alpha > 0;$
}
\KwOut{Distance transform $D$} 
\Begin{
	$B' \leftarrow (B_1, \ldots, B_K, B_1, \ldots, B_K, B_1, \ldots, B_K)$\; 
	$v' \leftarrow (v_1 - 2\pi, \ldots, v_K- 2\pi, v_1, \ldots, v_K, v_1+  2\pi, \ldots, v_K + 2\pi)$\; 
	$D' \leftarrow \mathrm{distTransReal}(B', v', \alpha)$\;
	$D \leftarrow (D'_{K+1}, \ldots, D'_{2K})$\; 
	\textbf{return} $D$\;
}
\end{algorithm}
\noindent We point out that this algorithm employs the real-valued distance transform of Section~\ref{subsec:DistTrafoReal}. The next result in particular shows that 	Algorithm~\ref{alg:distTransCirc} actually computes a minimizer of the distance transform \eqref{eq:distTransCirc}. The proof uses infimal convolutions on the real line
and employs the corresponding statement Theorem~\ref{thm:distTransReal} for real-valued data.
\begin{theorem}\label{thm:distTransCircle}
	Algorithm~\ref{alg:distTransCirc} computes  \eqref{eq:distTransCirc} in $O(K).$
\end{theorem}
\begin{proof}

First we observe that the arc length distance on $\mathbb{S}^1 = \mathbb T$
can be written using the absolute value on $(-\pi,\pi]$ by
\[
	d_\T(u, w) = \min \{ |u - 2\pi - w|; |u - w|; |u + 2\pi - w|\},
\]
for $u, w \in (-\pi, \pi].$ 
We define the extended real-valued functions $F,F'$ defined on $\mathbb R$ as follows:
we let $F(v_k) = B_k$ on the points $v_k$ and $F(r) = \infty$ for $r \in \R \setminus V$;
to define $F',$ we let 
\[
F'(t) = \min\left\{F(t-2\pi), F(t),F(t+2\pi)\right\}.
\]
Our goal is to show that $D_k$ is the infimal convolution of $F'$ and  $G$
with $G$ given by $G(v) = \alpha |v|.$ 
We get that
 \begin{align*}
  D_k =&  \min_{r \in \R} \{ F(r) + 
  \alpha\, \min \{ |r - 2\pi - v_k|; 
   |r - v_k|; |r + 2\pi - v_k|\} \} \\
    =&\min_{r\in \R} \min \{ F(r) +  \alpha |r - 2\pi - v_k|; \\
    &\qquad  F(r) +  \alpha  |r - v_k|; F(r) +  \alpha  |r + 2\pi - v_k|\} \} \\
    =& \min \{\min_{r \in \R} F(r + 2\pi) +  \alpha |r  - v_k|; \\
    & \qquad \min_{r \in \R} F(r) +  \alpha  |r - v_k|; \min_{r \in \R}F(r- 2\pi) +  \alpha  |r  - v_k|\} \} \\
    =&\min_{r \in \R} F'(r) +  \alpha  |r - v_k| = F' \square G(v_k).
    \end{align*}
Hence, $D_k$ is the infimal convolution of $F'$ and  $G.$
We now shift the vector of assumed values $v$ by $-2\pi$ and $2\pi$ and 
consider the concatenation with $v$ to obtain $v'$ which is given by  
 \[v' = (v_1 - 2\pi, \ldots ,v_K - 2\pi, v_1, \ldots ,v_K, v_1 + 2\pi, \ldots ,v_K + 2\pi).\]
We note that $v'$ is ordered ascendingly. 
We let
   \begin{equation}\label{eq:dl}
 D'_l = F' \square G(v'_l), \quad \text{for } l = 1,\ldots 3K.
  \end{equation}
  By Theorem \ref{thm:distTransReal}, 
  we can compute \eqref{eq:dl} in $\Oc(K)$
  using Algorithm~\ref{alg:distTrans}. 
Eventually, we observe that
 \begin{equation*}
  D_k = F' \square G(v_k) = D'_{K+k}, \quad \text{for } k = 1,\ldots, K,
 \end{equation*}
which completes the proof.
\end{proof}

\subsection{Complete algorithm}
The complete procedure is described 
in Algorithm~\ref{alg:l1tv}.
Summarizing, we have obtained the following result:
\begin{theorem}\label{thm:complexityGeneral}
Let $y \in \R^N$ and $V = \val(y),$
or $y \in \T^N$ and $V = \val(y) \cup \val(\tilde y).$
Further let $K$ be the number of elements in $V.$ 
Then Algorithm~\ref{alg:l1tv} 
computes a global minimizer of the
 $L^1$-TV problem with real-valued \eqref{eq:l1tv}
 or  circle-valued data \eqref{eq:tvCirc}
 in $\Oc(KN).$
 In particular, if data is quantized to a finite set,
 then the algorithms for real-valued or circle-valued signals are in $\Oc(N).$
\end{theorem}

\input{l1tvAlgo}

\section{Numerical results}\label{sec:numerics}

We illustrate the effects of $L^1$-TV
minimization for real and circle-valued data.
We consider both synthetic and real life data.

\begin{figure}
	\def\figfolderA{experiments/mapEstimation/}
		\def\figurewidth{0.47\columnwidth}
	\centering	
\centering
\includegraphics[width=\figurewidth]{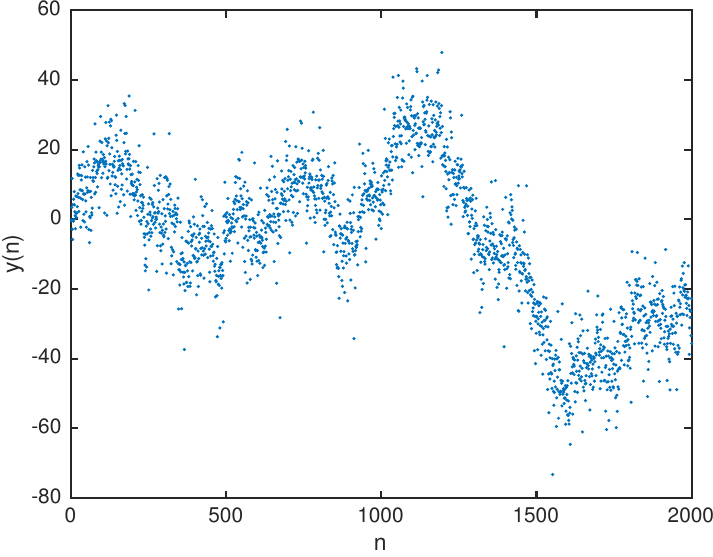} 
\hfill
\includegraphics[width=\figurewidth]{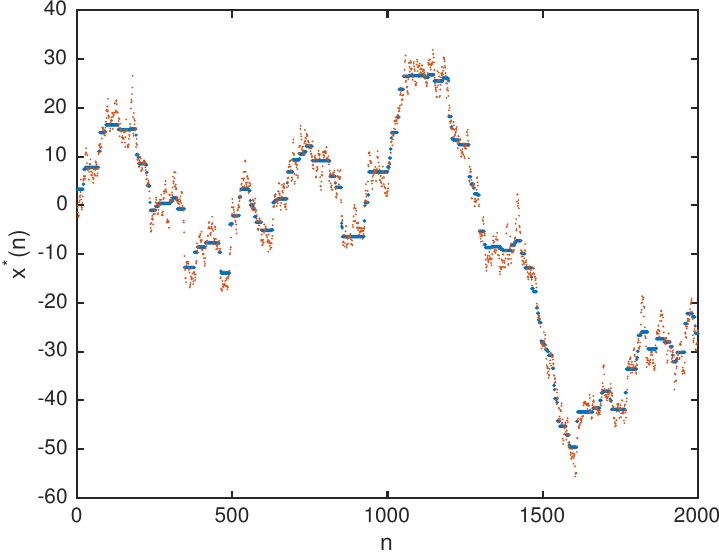} \\[2ex]

	\def\figfolderB{experiments/mapEstimationCirc/}
	\centering	
\centering
\includegraphics[width=\figurewidth]{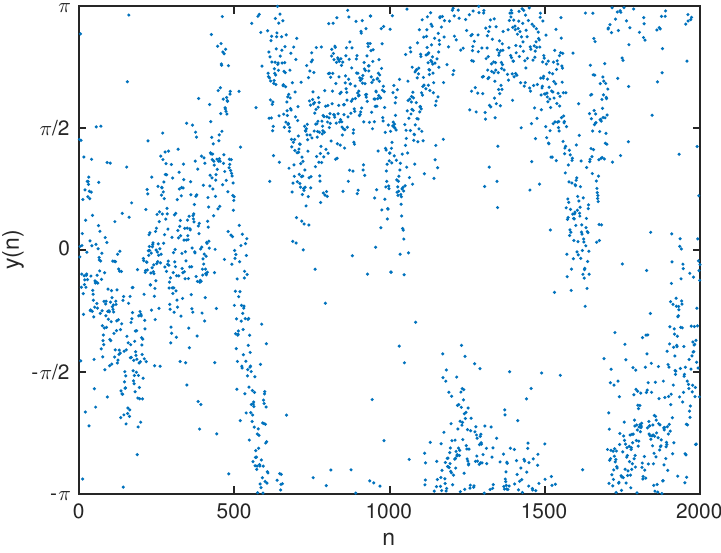} 
\hfill
\includegraphics[width=\figurewidth]{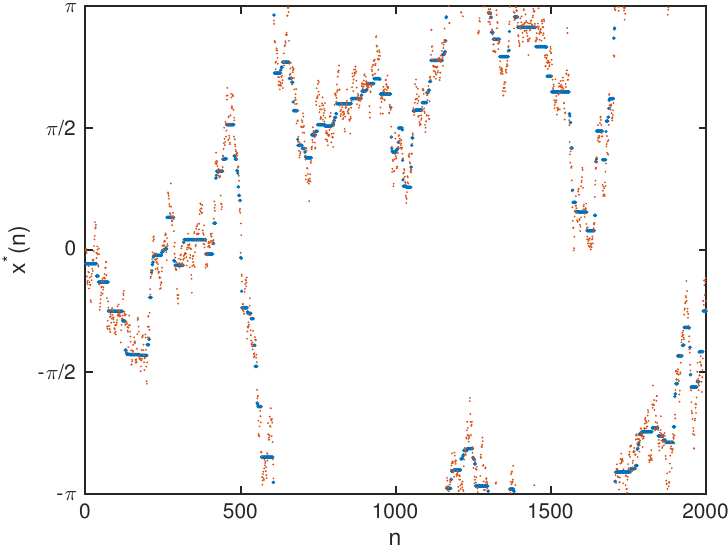} 

	\caption{\emph{Top left:} Realization of a Levy process with Laplacian increments ($\rho = \protect\input{\figfolderA eta.txt}$), corrupted with Laplacian white noise ($\xi = \xi_1 =  \ldots = \xi_N = \protect\input{\figfolderA xi.txt}$).
	\emph{Top right:} 
	The minimizer of the $L^1$-TV functional with parameter $\alpha = \xi/\rho$ is the maximum a posteriori estimate ($\deltaSNR$: $\protect\input{\figfolderA deltaSNR.txt},$ ground truth displayed as small red points).
	\emph{Bottom:} Analogous experiment for circle-valued data with $\rho = \protect\input{\figfolderB eta.txt},$ $\xi = \protect\input{\figfolderB xi.txt}$ 
	and $\alpha = \protect\input{\figfolderB alpha.txt}$
	($\deltaSNR$:~$\protect\input{\figfolderB deltaSNR.txt}$).
	}
\label{fig:synthReal}
\end{figure}

\paragraph{Experimental setup.}

We have implemented our algorithms in Matlab.
The experiments were conducted on a desktop computer 
with $3.5$ GHz Intel Xeon E5 and 32 GB memory.
The weight vector $w$ can be employed to account for 
non-equidistant sampling; it means that the data 
$y_n = f(t_n)$ is the sampling of a (continuously defined) signal $f$
at non-equidistant knots $t_1 < t_2 < \ldots < t_N.$
A reasonable choice for $w_n$ is the average distance of the sampling point $t_n$ to its nearest neighbors $t_{n-1}, t_{n+1},$
i.e., $w_n = (t_n - t_{n-1} + t_{n+1} - t_n)/2 = (t_{n+1} - t_{n-1})/2.$
In our experiments, we focus on equidistant sampling so that we have $w_n = 1$ for all $n = 1, \ldots, N.$
To quantify the denoising performance, we occasionally give
the manifold-valued version of the \emph{signal-to-noise ratio improvement} (see \cite[Chapter~10]{unser2014introduction} and \cite{WDS2013}).
It is given by
\[
	\deltaSNR = 10 \log_{10} \left(  \frac{\sum_{n} d(\bar{y}_{n}, y_{n})^2   }{\sum_{n} d(\bar{y}_{n}, x^*_{n})^2}\right),
\]
where $\bar y$ denotes the ground truth.
For real-valued data, we let $d$  denote the Euclidean metric.
If not mentioned explicitly, the regularization parameter $\alpha$ is adjusted empirically.
The higher we choose the value $\alpha$ the stronger we smooth the signal.

\paragraph{Circular $L^1$-TV on synthetic data.} 
In the introductory experiment (Figure~\ref{fig:synthCirc}), 
we have computed the total variation minimizer 
for a circle sample signal with known ground truth $\bar y$.
The signal $y$ was created by corrupting the phase angle $\bar \phi$
of the  original signal by Laplacian distributed white noise
of standard deviation $\sigma = 0.5.$
That is, the signal $y$ is given by  $y_j = e^{\mathrm{i}(\bar\phi_j + \eta_j)}$ where $\eta$ denotes 
the noise vector. 
The experiment illustrates
the denoising capabilities of total variation 
minimization for circle-valued data.
In particular, we observe that the phase jumps by $2\pi$ are taken 
into account properly.
The runtime was $\protect\input{experiments/randomSignal/runtime.txt}$ seconds.

\paragraph{$L^1$-TV as MAP estimator.}
Under certain assumptions 
on the signal and the noise,
 the Bayesian framework gives a suggestion for the parameter $\alpha.$
 For an introduction to the related statistical concepts we exemplarily refer to
 the book \cite{unser2014introduction}.
 
Assume that the true signal $\bar{y} \in \R^N$  (or $\bar{y} \in \T^N$)
 is generated according a Levy process with Laplacian increments; that is, $\bar{y}$ is a random vector
and the increments 
follow a distribution with density $P(\bar{y}_{n}| \bar{y}_{n-1}) \sim  e^{- d(\bar{y}_n, \bar{y}_{n-1})/\rho}.$
Also assume that the noise is distributed 
according to $P(y_n| \bar{y}_{n}) \sim e^{-d(y_n, \bar{y}_{n})/\xi_n}.$
Here, $\rho$ and $\xi_n $ are positive parameters.
The maximum a posteriori (MAP) estimator is
given by
\[
\begin{split}
	x^*_{MAP} &= \arg \max_{x} P(y|x) = \arg \max_{x} P(x)P(x|y) \\
	&= \arg \max_{x} \, \prod_{n=1}^{N-1} e^{-d(x_n, x_{n+1})/ \rho} \prod_{n=1}^N e^{-d(x_n, y_{n})/\xi_n } \\
	&= \arg \min_{x} \frac{1}{\rho} \sum_{n=1}^{N-1}   d(x_n, x_{n+1}) +  \sum_{n=1}^N \frac{1}{\xi_n} d(x_n, y_{n} ),
\end{split}
\]
The last equality has been obtained by taking the logarithm.
This derivation reveals that $L^1$-TV model is the MAP estimator 
in the  above probabilistic setup.
Therefore, the natural parameter choice is $\alpha = 1/\rho$ and $w_n = 1/\xi_n$ for $n = 1, \ldots, N.$
In particular, for uniform parameters $\xi=\xi_1 = \ldots = \xi_N,$ we have $\alpha = \xi/\rho.$
Figure \ref{fig:synthReal} shows the realization of such signals 
and their MAP estimates.

\paragraph{Robustness to impulsive noise.}

Besides the above Laplacian noise and innovation models,
the (real-valued) $L^1$-TV estimator
is known for its robustness to impulsive noise 
and for its good performance 
on piecewise constant signals; see, e.g., \cite{fu2006efficient,clason2009duality,dumbgen2009extensions,kolmogorov2015total}.
In the following,
we illustrate this observation,
and reveal a similarly good 
performance for circle  valued data with impulsive noise imposed.
We create a compound Poisson distributed random vector $s \in \R^N;$ 
that is, $s_n= 0$ with probability $e^{-\lambda}$ 
and $s_n$ is uniformly distributed in $[-a, a]$ with probability 
$1 - e^{-\lambda};$ see \cite{unser2014introduction}.
(Here, we use $a = 4.$)
The (true) signal $\bar y$ is given 
as the summation process of the increments $s;$ 
that is,  $\bar y_n = \sum_{j=1}^n s_j.$
Then, we corrupt the signal by impulsive noise 
which is also distributed according 
to a compound distribution with $\lambda'$
and $a' = \max_n |\bar y_n|.$
This means that $y_n = \bar y$ with probability $e^{-\lambda'}$
as well as that $y_n$ is uniformly distributed in $[-a,a]$ with probability $1- e^{-\lambda'}.$
In case of circle valued data,
we create a random signal in the same way with $a = a' = \pi.$
Then we consider the corresponding signal as phase angle.
Figure \ref{fig:synthRealImpulse} shows the 
realization of such stochastic processes and their $L^1$-TV estimates.

\begin{figure}
	\def\figfolderA{experiments/impulseNoise/}
		\def\figurewidth{0.47\columnwidth}
	\centering	
\centering
\includegraphics[width=\figurewidth]{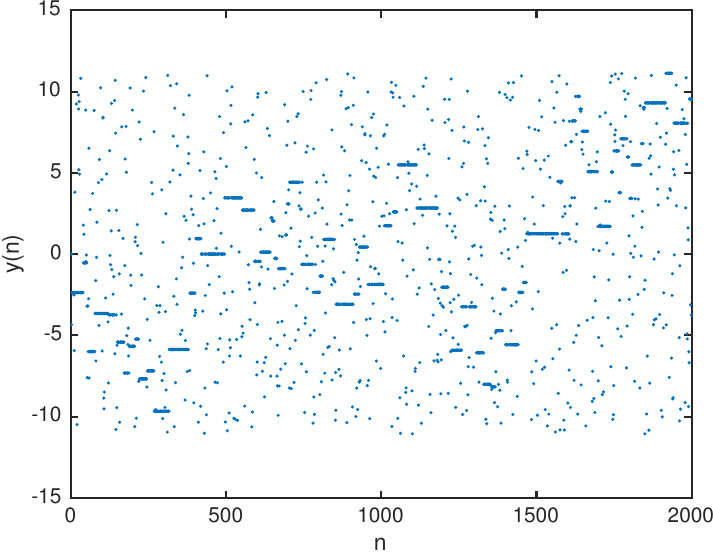} 
\hfill
\includegraphics[width=\figurewidth]{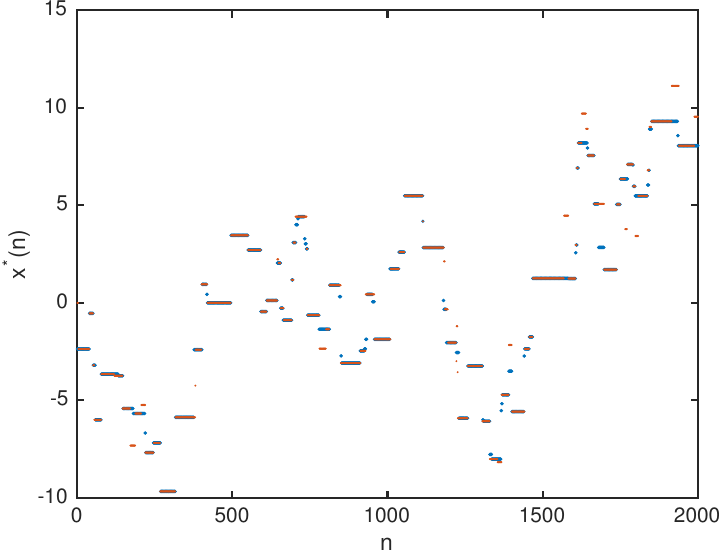} 
\\[2ex]
	\def\figfolderB{experiments/impulseNoiseCirc/}
		\centering	
\centering
\includegraphics[width=\figurewidth]{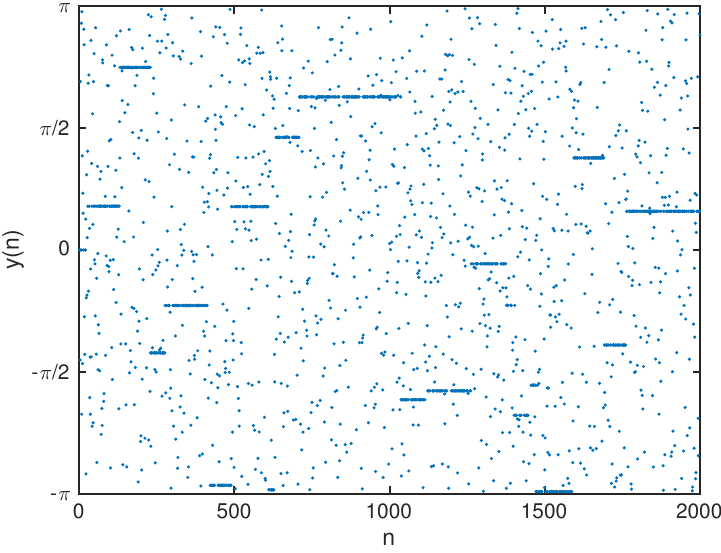} 
\hfill
\includegraphics[width=\figurewidth]{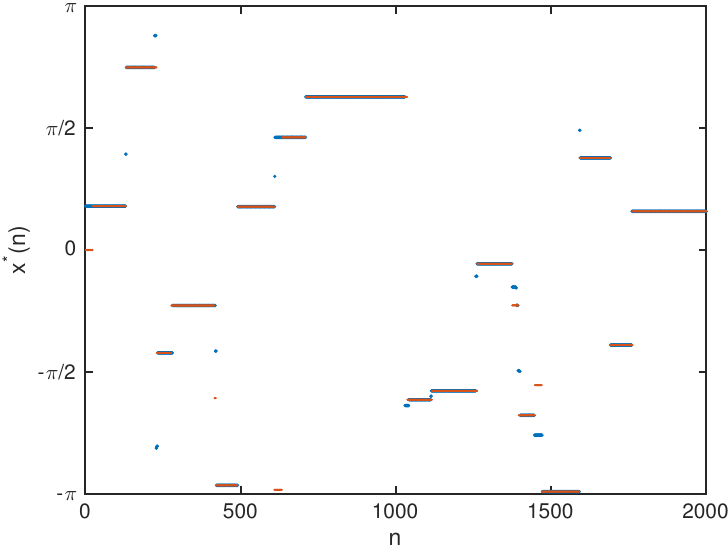} 

	\caption{\emph{Top left:} Realization of a (real-valued) Levy process with compound Poisson distributed increments ($\lambda = \protect\input{\figfolderA innoLambda.txt}$), corrupted with impulsive noise with compound Poisson distribution
		($\lambda' = \protect\input{\figfolderA noiseLambda.txt}$).
	\emph{Top~right:} The minimizer of the $L^1$-TV functional with parameter $\alpha = \protect\input{\figfolderA alpha.txt}$ ($\deltaSNR$: $\protect\input{\figfolderA deltaSNR.txt}$).
	\emph{Bottom:} Analogous experiment for circle-valued data with 	$\lambda = \protect\input{\figfolderB innoLambda.txt},$ $\lambda' = \protect\input{\figfolderB noiseLambda.txt}$ 
	and $\alpha = \protect\input{\figfolderB alpha.txt}$ ($\deltaSNR$:  $\protect\input{\figfolderB deltaSNR.txt}$).
	}
\label{fig:synthRealImpulse}
\end{figure}

\paragraph{Real life data -- Estimation of wind orientations.}
Next, we apply our algorithm 
to real life data.
The first data set consists of
wind directions
at the station WPOW1 (West Point, WA)
recorded every $10$ minutes in the year 2014.
The second data set consists of wind directions at the station VENF1 (Venice, FL) recorded every $60$ minutes in the same year.\footnote{Data available at \url{http://www.ndbc.noaa.gov/historical_data.shtml}.}
The data is given quantized to integer angles in degrees,
thus $K= 360.$
The regularized signal  
facilitates to identify the time intervals of approximately constant wind
direction.
For the estimate of the first data set (Figure~\ref{fig:wind}),
we observe a relatively regular and sudden alternation of the wind orientation
between around $0.5$ and $2.9$ radians every third to fifth day.
For the estimate of the second data set (Figure~\ref{fig:wind2}),
we observe an inclination towards the orientation angle $0.9$ radians
in the middle of the year, and to $0.8$ radians in the months of autumn.
Despite the lengths of the signals ($N=52543$ and $N = 8755$), 
the computational times amount to only around $20$ seconds and around $3$ seconds, respectively.

\begin{figure*}
\centering
	\def\figfolder{experiments/wind/}
		\def\figurewidth{1\textwidth}
\def\plottitle{}
\includegraphics[width=\figurewidth]{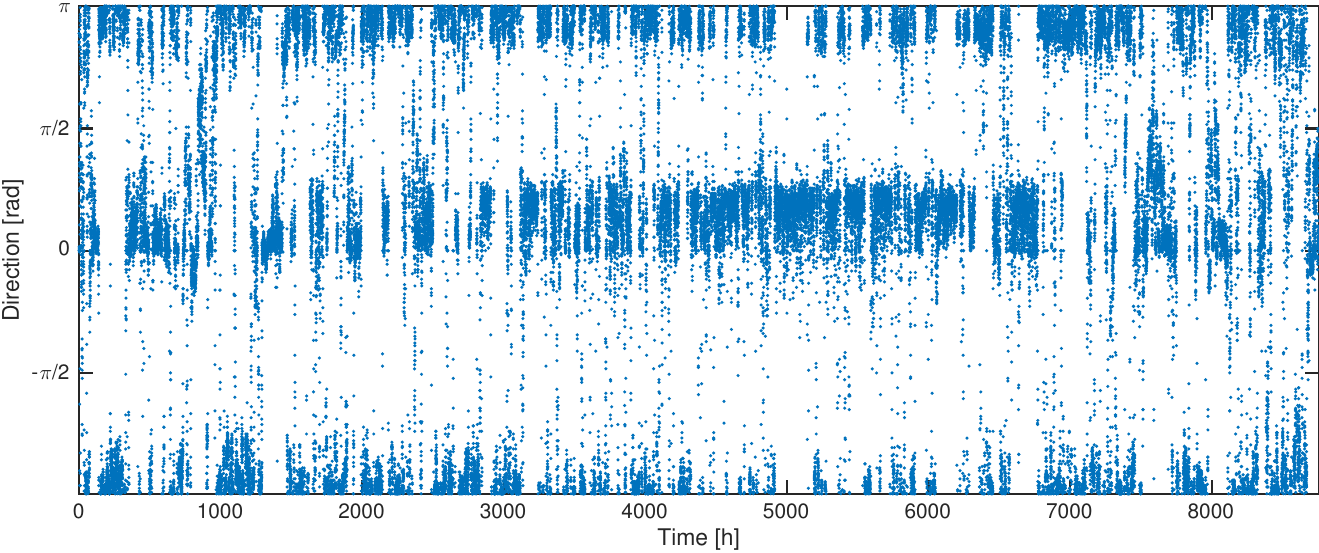} \\[4ex]
\includegraphics[width=\figurewidth]{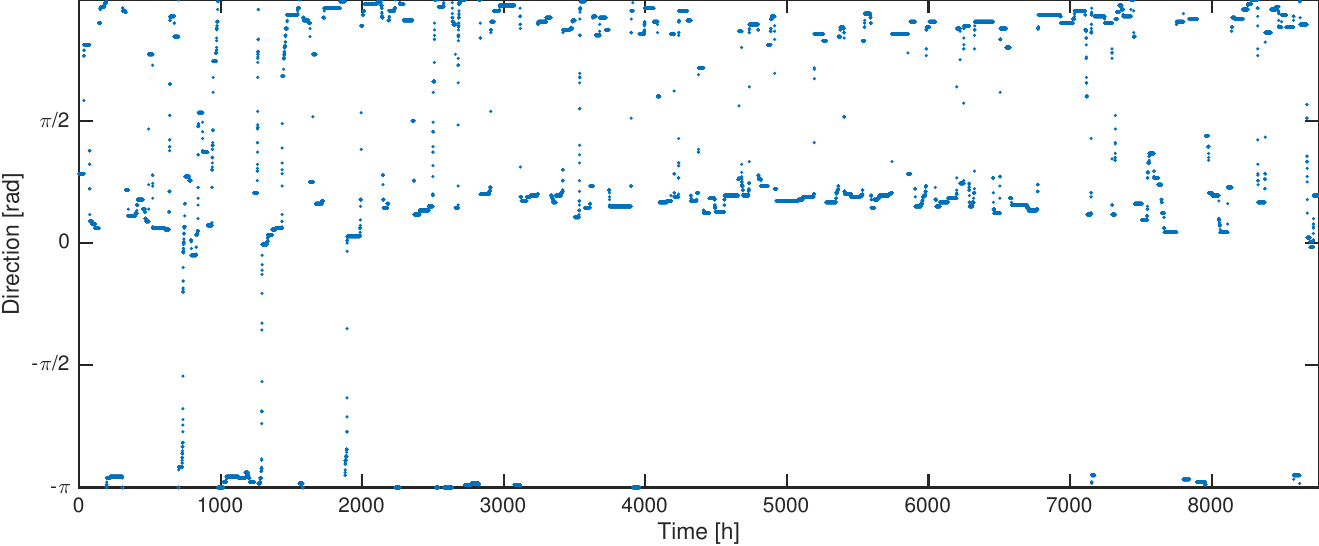}\\[2ex] 
	\caption{\emph{Top:} Wind directions 
	at Station WPOW1 (West Point, WA)
	recorded every 10 minutes in the year 2014.
	\emph{Bottom:} Total variation regularization with parameter
	$\alpha = \protect\input{\figfolder alpha.txt}.$ 
	The data is given quantized to $K = 360$ angles. 
 The time computation amounts to only $\protect\input{\figfolder runtime.txt}$ seconds 
	for the signal of length $N = \protect\input{\figfolder numel.txt}.$
	}
\label{fig:wind}
\end{figure*}

\begin{figure*}
\centering
	\def\figfolder{experiments/wind2/}
		\def\figurewidth{1\textwidth}
\def\plottitle{}
\includegraphics[width=\figurewidth]{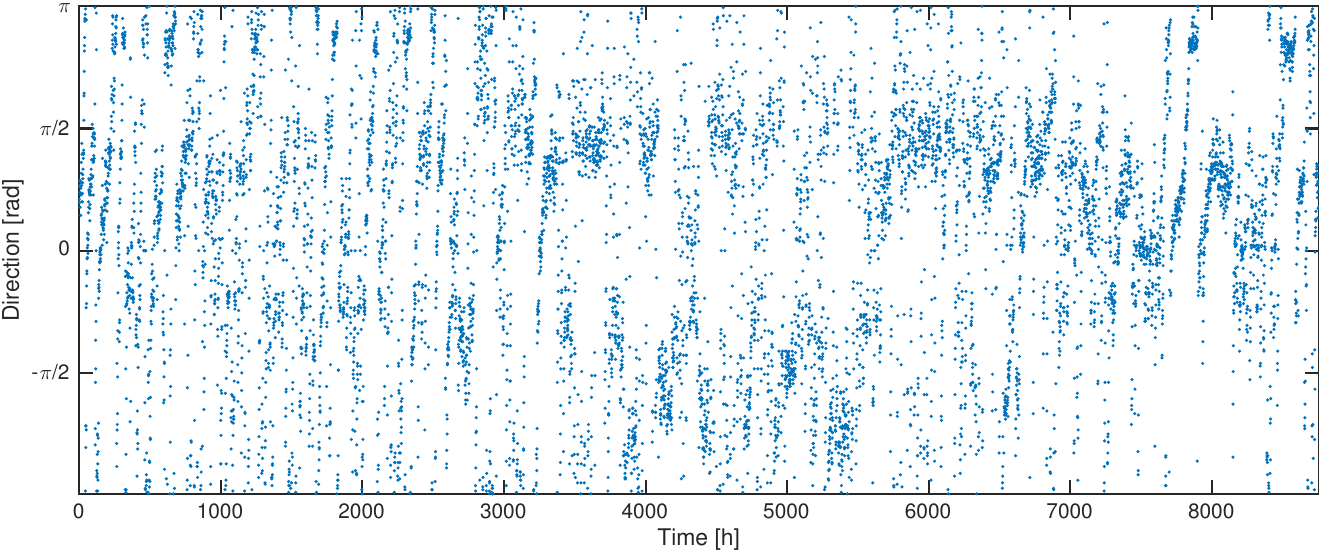} \\[4ex]
\includegraphics[width=\figurewidth]{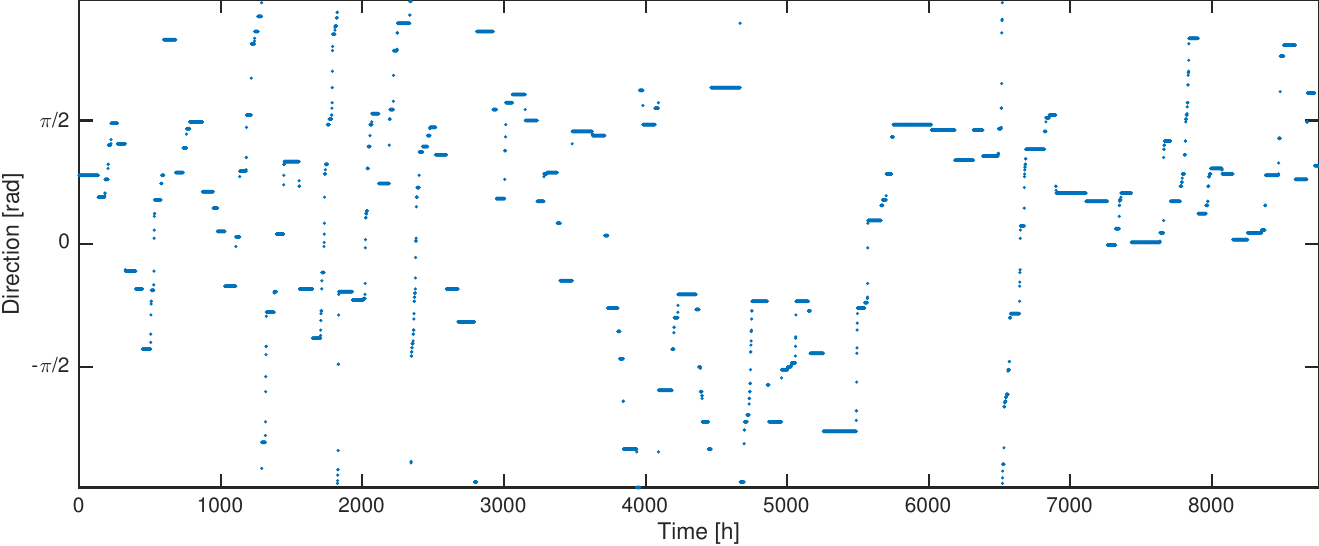}\\[2ex] 
	\caption{\emph{Top:} Wind directions at Station  VENF1 (Venice, FL)
	recorded every 60 minutes in the year 2014.
	\emph{Bottom:} Total variation regularization with parameter
	$\alpha = \protect\input{\figfolder alpha.txt}$ 
 ($N = \protect\input{\figfolder numel.txt},$ CPU time: $\protect\input{\figfolder runtime.txt}$ seconds).}
\label{fig:wind2}
\end{figure*}

\section{Discussion}

We have derived exact algorithms
for the $L^1$-TV problem with scalar
and circle-valued data. 
A first crucial point was the reduction 
of the search space to a finite set 
which allowed us 
to  employ the Viterbi algorithm.
The second key ingredient was a reduction
of the computational complexity based on 
a generalization of distance transforms.
The algorithms have quadratic  complexity in the  worst case.
The  complexity is linear
when the signal is quantized to a finite set.
We note that quantized signals 
appear frequently in practice,
for example in digitalized audio signals and images,
or when angular data is given in integer degrees 
as in the considered time series of wind directions.

The circular version is
the first exact solver for TV regularization of circle-valued signals.
Besides the application for jump-preserving denoising 
of angular signals,
it can also be used as building block 
for higher dimensional problems as in \cite{weinmann2014mumford}
or as benchmark for iterative strategies, e.g., for those of \cite{CS13,WDS2013,grohs2014total}.

Next we discuss the differences 
to previously proposed exact solvers for the real-valued case.
The solver of Dümbgen and Kovac \cite{dumbgen2009extensions}
is based 
on a generalization of the taut string algorithm combining
isotonic and antitonic regression functions
which is quite distinct from our approach.
The recent paper of Kolmogorov et al.~\cite{kolmogorov2015total}
seems at first glance to be related to our method
 because it also utilizes dynamic programming.
However, the strategy is fundamentally different: in \cite{kolmogorov2015total}, the algorithm is based on dynamically 
removing and appending breakpoints, 
whereas our method performs an efficient scanning  over the elements of 
the finite search space $V^N.$
The solvers of \cite{dumbgen2009extensions} and \cite{kolmogorov2015total}
have complexity  $\Oc(N \log N )$ and $\Oc(N \log \log N ),$ respectively.
Our method is competitive in terms 
of algorithmic complexity when 
the data is quantized which leads to the com\-plexity~$\Oc(N).$

The proposed approach 
appears to be unique for $L^1$ data terms.
In particular, we have provided counterexamples 
that the  utilized 
search space reduction
is not valid for quadratic data terms.
An exact and efficient algorithm 
for $L^2$-TV regularization of circle-valued signals
remains as an open question.

\section*{Acknowledgement}
Martin Storath and Michael Unser 
are supported by  the European Research Council under the European Union's Seventh Framework Programme (FP7/2007-2013) / ERC grant agreement no.~267439.
Andreas Weinmann is supported by the Helmholtz Association within the young investigator group VH-NG-526. 
Martin Storath and Andreas Weinmann  acknowledge the support by the DFG scientific network Mathematical Methods in Magnetic Particle Imaging.

{\small
\bibliographystyle{myplainnat}
\bibliography{l1tv}
}

\end{document}

%% file: symbols.tex
\usepackage{mathrsfs}

\renewcommand{\phi}{\varphi}

\renewcommand{\i}{i\,}

\newcommand{\N}{\mathbb{N}}

\newcommand{\R}{\mathbb{R}}
\renewcommand{\S}{\mathbb{S}}

\newcommand{\Ic}{\mathcal{I}}

\newcommand{\Oc}{\mathcal{O}}

%% file: customcommands.tex
\newcommand{\p}[1]{\left( #1 \right)}

\newcommand{\argmin}{\operatorname{\arg}\min}



%% file: experiments/wind2/alpha.txt
20

%% file: l1tvAlgo.tex

\begin{algorithm}[t]
\caption{Exact algorithm for the $L^1$-TV problem of real- or circle-valued signals}
\label{alg:l1tv}
\small
\SetKwInOut{Input}{Input}
\SetKwInOut{Output}{Output}
\SetKwInOut{Local}{Local}
\SetCommentSty{text}
\SetCommentSty{itshape}
\KwIn{Data $y \in  \R^N$ or $y \in  \T^N;$
regularization parameter $\alpha > 0$;  
weights $w \in (\R_0^+)^N$;
}
\KwOut{Global minimizer $\hat x$ of \eqref{eq:l1tv} or \eqref{eq:tvCirc};} 
\Begin{
\tcc{1. Init candidate values}
$V \leftarrow \val(y)$\tcc*{Real-valued case}
$V \leftarrow \val(y) \cup \val(\tilde y)$ \tcc*{Circle-valued case}
$v \leftarrow $ $K$-tuple of elements of $V,$ sorted ascendingly\;
\tcc{2. Tabulation}
\For{$k \leftarrow 1$ \KwTo $K$}{ 
		$B^1_k \leftarrow  w_1 d(v_k,y_1)$\; 
}
\For{$n \leftarrow 2$ \KwTo $N$}{ 
	$D \leftarrow \mathrm{distTransReal}(B^n, v, \alpha)$\tcc*{Real-valued case}
	$D \leftarrow \mathrm{distTransCirc}(B^n, v, \alpha)$\tcc*{Circle-valued case}
	\For{$k \leftarrow 1$ \KwTo $K$}{ 
		$B^n_k \leftarrow  w_n d(v_k,y_n) + D_k$\; 
	}
}
	
\tcc{3. Backtracking}
$l\leftarrow \argmin_{k=1,...,K} B^N_k$\;
$\hat x_n \leftarrow v_{l}$\;

\For{$n \leftarrow N-1, N-2,..., 1$}{ 

	    $l \leftarrow \argmin_{k=1,...,K}  B_k^n + \alpha\, d(v_k, \hat x_{n+1})$\;
	    $\hat x_n \leftarrow v_{l}$\;
	    }
\textbf{return} $\hat x$\;
}
\end{algorithm}

%% file: experiments/mapEstimation/eta.txt
1.0

%% file: experiments/mapEstimation/xi.txt
5.0

%% file: experiments/impulseNoise/deltaSNR.txt
17.6

%% file: experiments/impulseNoiseCirc/alpha.txt
10.0

%% file: experiments/impulseNoiseCirc/deltaSNR.txt
13.4

%% file: experiments/randomSignal/runtime.txt
3.5

%% file: experiments/impulseNoise/innoLambda.txt
0.1

%% file: experiments/impulseNoise/noiseLambda.txt
0.5

%% file: experiments/impulseNoise/alpha.txt
5.0

%% file: experiments/impulseNoiseCirc/innoLambda.txt
0.01

%% file: experiments/impulseNoiseCirc/noiseLambda.txt
1.0

%% file: experiments/wind2/runtime.txt
3.3

%% file: experiments/wind2/numel.txt
8755